\documentclass[nointlimits]{amsart}

\usepackage[T1]{fontenc} 
\usepackage[utf8]{inputenc} 
\usepackage[USenglish]{babel} 

\usepackage{lmodern}
\usepackage{mathrsfs}
\usepackage{microtype}

\usepackage{url}\usepackage{natbib}

\usepackage{amsmath, amssymb, amsthm}

\usepackage{soul}

\usepackage{xcolor}

\usepackage[normalem]{ulem}

\usepackage[shortlabels]{enumitem}

\usepackage{hyperref}
\newcommand{\TITLE}{Operators on Orlicz sequence spaces and \texorpdfstring{$\Delta_2$-fundamentality}{Δ2-fundamentality}}
\hypersetup{
	unicode=true,
	breaklinks=true,
	pdftitle={\TITLE},
	pdfauthor={David Edmunds, Jan Krejčí, Zdeněk Mihula, and Luboš Pick}
	}

\makeatletter
\let\save@mathaccent\mathaccent
\newcommand*\if@single[3]{%
  \setbox0\hbox{${\mathaccent"0362{#1}}^H$}%
  \setbox2\hbox{${\mathaccent"0362{\kern0pt#1}}^H$}%
  \ifdim\ht0=\ht2 #3\else #2\fi
  }
\newcommand*\rel@kern[1]{\kern#1\dimexpr\macc@kerna}
\newcommand*\widebar[1]{\@ifnextchar^{{\wide@bar{#1}{0}}}{\wide@bar{#1}{1}}}
\newcommand*\wide@bar[2]{\if@single{#1}{\wide@bar@{#1}{#2}{1}}{\wide@bar@{#1}{#2}{2}}}
\newcommand*\wide@bar@[3]{%
  \begingroup
  \def\mathaccent##1##2{%
    \let\mathaccent\save@mathaccent
    \if#32 \let\macc@nucleus\first@char \fi
    \setbox\z@\hbox{$\macc@style{\macc@nucleus}_{}$}%
    \setbox\tw@\hbox{$\macc@style{\macc@nucleus}{}_{}$}%
    \dimen@\wd\tw@
    \advance\dimen@-\wd\z@
    \divide\dimen@ 3
    \@tempdima\wd\tw@
    \advance\@tempdima-\scriptspace
    \divide\@tempdima 10
    \advance\dimen@-\@tempdima
    \ifdim\dimen@>\z@ \dimen@0pt\fi
    \rel@kern{0.6}\kern-\dimen@
    \if#31
      \overline{\rel@kern{-0.6}\kern\dimen@\macc@nucleus\rel@kern{0.4}\kern\dimen@}%
      \advance\dimen@0.4\dimexpr\macc@kerna
      \let\final@kern#2%
      \ifdim\dimen@<\z@ \let\final@kern1\fi
      \if\final@kern1 \kern-\dimen@\fi
    \else
      \overline{\rel@kern{-0.6}\kern\dimen@#1}%
    \fi
  }%
  \macc@depth\@ne
  \let\math@bgroup\@empty \let\math@egroup\macc@set@skewchar
  \mathsurround\z@ \frozen@everymath{\mathgroup\macc@group\relax}%
  \macc@set@skewchar\relax
  \let\mathaccentV\macc@nested@a
  \if#31
    \macc@nested@a\relax111{#1}%
  \else
    \def\gobble@till@marker##1\endmarker{}%
    \futurelet\first@char\gobble@till@marker#1\endmarker
    \ifcat\noexpand\first@char A\else
      \def\first@char{}%
    \fi
    \macc@nested@a\relax111{\first@char}%
  \fi
  \endgroup
}
\makeatother

\numberwithin{equation}{section}

\theoremstyle{plain}
\newtheorem{theorem}{Theorem}[section]

\newtheorem{lemma}[theorem]{Lemma}

\theoremstyle{definition}
\newtheorem{definition}[theorem]{Definition}

\newcommand{\R}{\mathbb{R}}
\newcommand{\rn}{\R^n}

\newcommand{\N}{\mathbb{N}}

\newcommand{\DTwo}{\Delta_2}
\newcommand{\DTwoZ}{\DTwo^0}
\newcommand{\tA}{\tilde{A}}

\newcommand{\ttA}{\tilde{\tA}}
\newcommand{\lA}{\ell_A}
\newcommand{\ltA}{\ell_{\tA}}
\newcommand{\lAOrl}{\ell_{(A)}}
\newcommand{\ltAOrl}{\ell_{(\tA)}}
\newcommand{\rA}{\varrho_A}
\newcommand{\rtA}{\varrho_{\tilde{A}}}

\def\loc{\operatorname{loc}}
\def\rn{\mathbb R^n}

\newcommand*\dd{\mathop{}\!\mathrm{d}}

\title{\TITLE}
\author{David E. Edmunds}
\address{David E. Edmunds, Department of Mathematics, University of Sussex, Pevensey 2 Building, BN1 9QH Brighton, Sussex, United Kingdom}
\email{davideedmunds@aol.com}

\author{Jan Krejčí}
\address{Jan Krejčí, Department of Mathematical Analysis,
Faculty of Mathematics and Physics,
Charles University,
So\-ko\-lo\-vsk\'a~83,
186~75 Praha~8,
Czech Republic}
\email{jankrejci.gmk@gmail.com}

\author{Zden\v ek Mihula}
\address{Zden\v ek Mihula, Czech Technical University in Prague, Faculty of Electrical Engineering, Department of Mathematics, Technick\'a~2, 166~27 Praha~6, Czech Republic}
\email{mihulzde@fel.cvut.cz}
\urladdr{\href{https://orcid.org/0000-0001-6962-7635}{0000-0001-6962-7635}}

\author{Lubo\v{s} Pick}
\address{Lubo\v{s} Pick, Department of Mathematical Analysis,
Faculty of Mathematics and Physics,
Charles University,
So\-ko\-lo\-vsk\'a~83,
186~75 Praha~8,
Czech Republic}
\email{pick@karlin.mff.cuni.cz}
\urladdr{\href{https://orcid.org/0000-0002-3584-1454}{0000-0002-3584-1454}}

\begin{document}
\setcitestyle{numbers}
\bibliographystyle{abbrv}


\begin{abstract}
A classical result states that the Hardy--Littlewood maximal operator is bounded on an Orlicz space $L^A(\rn)$ if and only if its conjugate Young function $\tA$ satisfies the $\Delta_2$-condition. The same condition also characterizes the boundedness on $L^A(0,\infty)$ of the Hardy averaging operator. We consider a discrete analogue of the problem, extended to a general interpolation framework. We offer several characterizing conditions for the boundedness of discrete maximal and average operators on Orlicz spaces. Although the principal result is as expected, for its proof some new techniques have to be developed. To this end, we  introduce a new notion of the so-called $\Delta_2$-fundamental sequence, and give its interesting characterization by a simple condition involving only a limes superior of the ratio of two subsequent terms. We also prove a dual statement concerning operators of Copson type. 
\end{abstract}

\maketitle


\section{Introduction}

It is classically known that the original one-dimensional Hardy--Littlewood maximal operator 
\begin{equation*}
    \theta f(t) =
    \sup_{|t-s|>0}
    \frac{1}{t-s}\int_{s}^{t}
    |f(y)|\,dy,
\end{equation*}
defined for every locally integrable function on $(0,\infty)$ and every $t\in(0,\infty)$, is bounded on a rearrangement-invariant space $E$ if and only if the norm of the dilation operator $D_t$, given by
\begin{equation*}
    D_tf(s)=f(t^{-1}s)\quad\text{for $t,s\in(0,\infty)$,}
\end{equation*}
satisfies
\begin{equation*}
    \|D_t\|_{E}=o(t)\quad\text{as $t\to\infty$.}
\end{equation*}
For this result, as well as for the precise definition of rearrangement-invariant spaces and other details, we refer the reader to~\cite[Chapter II, Theorem 6.10]{Kre:82}. The boundedness on rearrangement-invariant spaces of the higher-dimensional modification of the Hardy--Littlewood maximal operator,  given by
\begin{equation*}
    Mf(x) =
    \sup_{Q\owns x}
    \frac{1}{|Q|}\int_{Q}
    |f(y)|\,dy,
\end{equation*}
in which the supremum is extended over all $n$-dimensional cubes $Q$ with sides parallel to the coordinate axes, $f\in L^{1}_{\loc}(\rn)$, and $x\in\rn$,
can further be characterized in terms of indices, see~\cite{Lor:55,Shi:65,Boy:67}.
Gallardo~\cite{Gal:88} characterized the class of all Orlicz spaces over an ambient Euclidean space on which this operator is bounded, and, at the same time, noted that it coincides with the collection of all Orlicz spaces over an interval on which the Hardy average operator is bounded.

Our goal is  to consider analogous problems for discrete versions of the mentioned  operators. 
More precisely, we will give a characterization of all Orlicz sequence spaces on which the discrete Hardy--Littlewood maximal operator, still denoted  $M$, and given as
\begin{equation*}
    (Mx)_n = \sup_{1\leq i\leq n \leq j} \frac1{j - i + 1}\sum_{k = i}^j |x_k|\quad\text{for $x=\{x_n\}_{n = 1}^\infty$ and $n\in\N$},
\end{equation*}
is bounded. Similarly as in the continuous case, we shall consider also the discrete Hardy operator $P$, given as
\begin{equation*}
    (Px)_n = \frac1{n} \sum_{k = 1}^n x_k\quad\text{for $x=\{x_n\}_{n = 1}^\infty$ and $n\in\N$,}
\end{equation*}
and its dual operator with respect to the $\ell_1$-pairing, namely the Copson operator $Q$, given as
\begin{equation*}
    (Qx)_n = \sum_{k = n}^\infty \frac{x_k}{k}\quad\text{for $x=\{x_n\}_{n = 1}^\infty$ and $n\in\N$}.
\end{equation*}
Note that
\begin{equation*}
    |(Px)|_n=|(Px)_n| \leq (P|x|)_n \leq
    (Mx)_n \quad \text{for every $n\in\N$}.
\end{equation*}

Our first main result gives a characterization of boundedness of both $M$ and $P$ by several criteria.
We will denote by $\DTwoZ$ the class of Young functions that satisfy the $\Delta_2$ condition near zero. For precise definitions and further details see Section~\ref{sec:prel}.

\begin{theorem}
    \label{T:main}
Let $A$ be a Young function and $\tA$ its complementary function. The following statements are equivalent.
\begin{enumerate}[(1)]
    \item\label{it:interpolation_modular} For every quasi-linear operator $T$, defined on $\ell_\infty$ with values in sequences, satisfying
    \begin{equation}\label{E:interpolation:endpoints}
        T\colon \ell_1 \to \ell_{1,\infty} \qquad\text{and}\qquad T\colon \ell_\infty \to \ell_\infty \qquad \text{boundedly},
    \end{equation}
    there are constants $M,\gamma>0$ such that
\begin{equation}\label{E:interpolation:modular}
    \sum_{n=1}^\infty A(|Tx|_n) \leq M \sum_{n=1}^\infty A(\gamma |x_n|) \quad \text{for every $x=\{x_n\}_{n=1}^\infty\in\lA$}.
\end{equation}
    \item\label{it:interpolation_norm} Every quasi-linear operator $T$, defined on $\ell_\infty$ with values in sequences, satisfying the endpoint estimates \eqref{E:interpolation:endpoints}
    is bounded on $\lA$.
    \item\label{it:average_bdd} The averaging operator $P$ is bounded on $\lA$.
    \item\label{it:maximal_bdd} The maximal operator $M$ is bounded on $\lA$.
    \item\label{it:maximal_bdd_restricted} There is a constant $C>0$ such that
    \begin{equation}\label{E:maximal_bdd_restricted}
        \|M\chi_E\|_{\lA} \leq C \|\chi_E\|_{\lA} \quad \text{for every finite $E\subseteq \N$}.
    \end{equation}
    \item\label{it:delta20} $\tA\in\DTwoZ$.
    \item\label{it:every_fund_seq} $\tA$ is positive outside the origin and for every sequence $\{t_n\}_{n=1}^\infty$ satisfying
    \begin{equation}\label{E1}
        t_n \searrow0 \quad \text{as $n\to\infty$}
    \end{equation}
    and
    \begin{equation}\label{E2}
        \limsup_{n\to\infty} \frac{t_n}{t_{n+1}} \leq 2,
    \end{equation}
    we have
    \begin{equation}\label{E3}
        \sup_{n\in\N} \frac{\tA(2t_n)}{\tA(t_n)} < \infty.
    \end{equation}
     \item\label{it:one_fund_seq} $\tA$ is positive outside the origin, and there is a sequence $\{t_n\}_{n=1}^\infty$ satisfying \eqref{E1} and \eqref{E2},
    for which \eqref{E3} is true.
\end{enumerate}
\end{theorem}

The (not really unexpected) characterization of the boundedness of $M$ on an Orlicz sequence space $\ell_{A}$, where $A$ is a Young function, by the localized variant (near zero) of the $\Delta_2$ condition, is the equivalence of statements~\ref{it:maximal_bdd} and~\ref{it:delta20}. Likewise, the theorem yields the coincidence of the action of $M$ and $P$ on Orlicz sequence spaces, which also has its analogy in the continuous world and which was pointed out by Gallardo in~\cite{Gal:88}. The same goes for the inequality in the modular form~\eqref{E:interpolation:modular} and the inequality restricted to characteristic functions~\eqref{E:maximal_bdd_restricted}, both observed in the Euclidean version by Gallardo. However, in accordance with the notoriously known fact that the discrete and continuous worlds are quite different from one another and new techniques are required in order to establish results in the former, we discover some stuff which does not have analogues for continuous versions of the notions and which has not been noticed before. This is reflected by the last two  statements of Theorem~\ref{T:main}.

In the course of the proof of Theorem~\ref{T:main}, we make use of a construction which is an inspiration for introducing the following interesting new notion.

\begin{definition}
    A sequence  $\{t_n\}_{n=1}^\infty\subset (0,\infty)$ satisfying $t_n\searrow 0$ is called \emph{$\DTwoZ$-fundamental} if for every Young function $A$ positive outside the origin and such that
    \begin{equation*}
        \sup_{n\in\N} \frac{A(2t_n)}{A(t_n)} < \infty,
    \end{equation*}
    one has $A\in\DTwoZ$.
\end{definition}

Since the $\DTwoZ$-fundamentality of a sequence is a crucial feature in the proof of the main result, it is worth studying in detail. We will now show that there is a rather surprising simple characterization of it. A remarkable fact is that it is universal for all Young functions. We believe that the following theorem is of independent interest in the general theory of Orlicz spaces.

\begin{theorem}\label{T:fundamental-sequence}
    A sequence $\{t_n\}_{n=1}^\infty\subset (0,\infty)$ satisfying $t_n\searrow 0$ is $\DTwoZ$-fundamental if
    and only if it satisfies~\eqref{E2}.
\end{theorem}

We predict that the $\Delta_2$-fundamentality will turn out to have its analogue in the continuous setting, too, but we do not pursue this line of research in this paper.

\medskip

We shall now consider the operator $Q$. Interestingly, it turns out that a considerably different statement that captures its initial endpoint behavior is needed, compared to the analogous result for $P$  and $M$.
\begin{theorem}
    \label{T:main_dual}
Let $A$ be a Young function and $\tA$ its complementary function. The following statements are equivalent.
\begin{enumerate}[(1)]
    \item\label{it:interpolation_dual} Every quasi-linear operator $T$, defined on sequences $x=\{x_n\}_{n=1}^\infty$ for which $(Q|x|)_1< \infty$ with values in sequences, satisfying for all $N\in\N$ and $x=\{x_n\}_{n=1}^\infty$ for which $(Q|x|)_1< \infty$,
    \begin{equation}\label{E:interpolation_dual:HLP}
        \sum_{k=1}^N (Tx)_k^* \leq C \sum_{k=1}^N (Q|x|)_k
    \end{equation}
    with some constant $C>0$ is bounded on $\lA$.
    \item\label{it:average_bdd_dual} The Copson operator $Q$ is bounded on $\lA$.
    \item\label{it:delta20_dual} $A\in\DTwoZ$.
\end{enumerate}
\end{theorem}

We finally turn our attention to the discrete Hilbert transform, defined by
    \begin{equation*}
        (Hx)_n = \sum_{k\in\mathbb Z\setminus\{0\}}
        \frac{x_{n-k}}{k},
        \quad\text{for $x=\{x_n\}_{n = -\infty}^\infty$ and $n\in\mathbb Z$.}
    \end{equation*}
It has been of interest for a long time (cf.~e.g.~\cite[Chapter IX]{Har:88} or the discussion in~\cite[page~19]{Ben:80}). This operator significantly differs from those considered so far in the sense that the sum is extended over all integers (and not solely the positive ones), but the results obtained can be used to give a characterization of its boundedness on Orlicz spaces nonetheless. Let us still recall that the boundedness of various modifications of the integral version of the Hilbert transform (defined on $\R$ or localized to a torus etc.) have been treated for instance in~\cite{Rya:63} or~\cite{Boy:67}.

\begin{theorem}
    \label{T:hilbert}
    Let $A$ be a Young function and $\tA$ its complementary function. The discrete Hilbert transform $H$ satisfies 
    \begin{equation}\label{E:h-on-orlicz}
        H\colon \ell_A\to\ell_A
    \end{equation}
    if and only if both $A$ and $\tA$ satisfy the condition $\DTwoZ$.
\end{theorem}

The paper is structured as follows. We first collect the background material in the preliminary section, and then, in the last section, we give detailed proofs of all the results.
    
\section{Preliminaries}\label{sec:prel}

We say that a function $A\colon[0,\infty]\to[0,\infty]$ satisfying
$A(0)=0$ and not being constant on the entire interval $(0,\infty)$ is a \emph{Young
function} if it is left continuous and convex on $[0,\infty)$. A Young function $A$ that is positive and finite-valued on the interval $(0, \infty)$ and satisfies
\begin{equation*}
    \lim_{t\to0^+}\frac{A(t)}{t} = 0 \qquad \text{and} \qquad \lim_{t\to\infty}\frac{A(t)}{t} = \infty
\end{equation*}
is called an \emph{$N$-function}. For every Young function $A$, the function $t\mapsto \frac{A(t)}{t}$ is nondecreasing on $(0,\infty)$. The convexity of any Young function $A$ implies that
\begin{align}
    A(\lambda t) &\leq \lambda A(t) \quad \text{for all $\lambda\in[0,1]$ and $t\in[0, \infty)$} \label{prel:convexity_small_lambda}\\
    \intertext{and}
    A(\lambda t) &\geq \lambda A(t) \quad \text{for all $\lambda\in[1, \infty)$ and $t\in[0, \infty)$}. \nonumber
\end{align}

Given a sequence $x=\{x_n\}_{n=1}^\infty$, we denote by $|x|$ the sequence $\{|x_n|\}_{n = 1}^\infty$. Furthermore, when $x=|x|$, we write $x\geq0$.

Given a Young function $A$, we define the \emph{Orlicz semimodular} $\rA$ as
\begin{equation*}
    \rA(x) = \sum_{n = 1}^\infty A(|x_n|),\ x=\{x_n\}_{n=1}^\infty.
\end{equation*}
The \emph{Orlicz sequence space} $\lA$ is then defined as
\begin{equation*}
	\lA = \{x=\{x_n\}_{n=1}^\infty: \|x\|_{\lA} < \infty\},
\end{equation*}
where
\begin{equation*}
	\|x\|_{\lA} =
		\inf\Big\{
			\lambda>0: \rA\Big(\frac{x}{\lambda}\Big) \le 1 \Big\},\ x=\{x_n\}_{n=1}^\infty.
\end{equation*}
The Orlicz sequence space $\lA$ endowed with the functional $\|\cdot\|_{\lA}$, which is called the \emph{Luxemburg functional}, is a rearrangement-invariant Banach space in the sense of \cite{BS}. The rearrangement-invariance means that
\begin{equation}\label{prel:Lux_norm_ri}
    \|x\|_{\lA} = \|x^*\|_{\lA},
\end{equation}
where $x^*$ is the \emph{nonincreasing rearrangement of a sequence} $x=\{x_n\}_{n=1}^\infty$ defined as
\begin{equation*}
    x^*_n = \inf\{\lambda>0: \#\{n\in\N: |x_n| > \lambda \}\leq n - 1\},\ n\in\N.
\end{equation*}
When $x\in c_0$, its nonincreasing rearrangement is simply the sequence $\{|x_n|\}_{n=1}^\infty$ rearranged in nonincreasing order. There are many useful inequalities concerning rearrangements (see~\cite{Har:88} for some classical ones). One of them is the \emph{Hardy--Littlewood inequality}:
\begin{equation*}
    \sum_{n = 1}^\infty |x_n||y_n| \leq \sum_{n = 1}^\infty x^*_n y^*_n \quad \text{for all $x=\{x_n\}_{n=1}^\infty$, $y=\{y_n\}_{n=1}^\infty$}.
\end{equation*}
In particular, it implies that
\begin{equation}\label{prel:HL_on_P}
    (P|x|)_n \leq (P(x^*))_n \quad \text{for all $x$ and $n\in\N$},
\end{equation}
in which the latter, being an average of a nonincreasing sequence, is nonincreasing.

It is sometimes useful to know that the infimum in the definition of $\|\cdot\|_{\lA}$ is attained. More precisely, for every $x\in\lA\setminus\{0\}$, we have
\begin{equation}\label{prel:inf_att_by_Lux_norm}
    \rA\Big( \frac{x}{\|x\|_{\lA}} \Big) \leq 1.
\end{equation}

The Lebesgue sequence spaces $\ell_p$ for $p\in[1, \infty]$ are Orlicz spaces. More precisely, we have $\ell_p = \ell_{A_p}$ with equal norms, where
\begin{align*}
    A_p(t) = t^{p} \quad \text{if $p\in[1, \infty)$} \qquad \text{and} \qquad A_p(t) = \infty\chi_{(1, \infty]}(t) \quad \text{if $p=\infty$}
\end{align*}
for $t\in[0, \infty)$. Furthemore, every Orlicz sequence space $\lA$ satisfies
\begin{equation*}
    \ell_1 \hookrightarrow \lA \hookrightarrow \ell_\infty.
\end{equation*}
The symbol $\hookrightarrow$ stands for a continuous embedding. Given two (quasi-)normed sequence spaces $A$ and $B$, we write $A\hookrightarrow B$ if $A\subseteq B$ in the set-theoretical sense and there is a constant $C>0$ such that $\|x\|_B \leq C \|x\|_A$ for every $x \in A$. When $A\hookrightarrow B$ and $B\hookrightarrow A$ simultaneously, we write $A = B$.

A function $A\colon[0, \infty] \to[0, \infty]$ is a Young function if and only if there exists a nondecreasing function $a\colon[0,\infty)\to[0,\infty]$, which is neither constantly $0$ nor constantly $\infty$ on $(0,\infty)$, such that $a(0) = 0$ and
\begin{equation} \label{prel:orlicz-integral-representation-of-young-function}
	A(t)=\int_0^t a(\tau) \dd\tau
		\quad\text{for every $t\in[0, \infty]$}.
\end{equation}
Since $a$ is monotone, it can be assumed to be left-continuous or right-continuous, as necessary. It follows immediately from the monotonicity of $a$ that
\begin{equation}
\label{E:relation-between-a-and-A}
    \frac{A(t)}{t} \le a(t) \quad \text{for every $t \in(0,\infty)$.}
\end{equation}

Furthermore, the Orlicz semimodular $\rA$ can be expressed in terms of the density $a$ as
\begin{equation}\label{prel:Orlicz_mnodular_via_distribution}
    \rA(x) = \int_0^\infty a(t) \#\{n\in\N: |x_n| > t \} \dd t \quad \text{for every $x=\{x_n\}_{n=1}^\infty$}.
\end{equation}

We say that a Young function $A$ satisfies the \emph{$\DTwo$ condition} (globally), and write $A\in\DTwo$, if there is a constant $C>0$ such that
\begin{equation*}
    A(2t) \leq C A(t) \quad \text{for every $t\in(0, \infty)$}.
\end{equation*}
Note that since $A$ is not constant on the interval $(0, \infty)$, $A$ is necessarily positive and finite-valued on the interval $(0, \infty)$ if $A\in\DTwo$. Moreover, every finite-valued Young function $A$ is locally Lipschitz continuous on the interval $[0, \infty)$. 
Furthermore, we say that a Young function $A$ satisfies the \emph{$\DTwo$ condition near $0$}, and write $A\in\DTwoZ$, if $A$ is positive on the interval $(0, \infty)$ and there are $t_0\in(0, \infty)$ and a constant $C>0$ such that
\begin{equation*}
    A(2t) \leq C A(t) \quad \text{for every $t\in(0, t_0]$}.
\end{equation*}
Note that $A\in\DTwoZ$ entails that $A$ is positive outside $0$. 

Let $A$ be a Young function. The function $\tilde A\colon[0,\infty]\to[0,\infty]$ defined as
\begin{equation*}
	\tA(t)=\sup_{\tau\in[0, \infty)}\big( \tau t-A(\tau) \big),\ t\in[0,\infty],
\end{equation*}
is called the \emph{complementary function} of $A$. The complementary function of a Young function is also a Young function. Moreover, if $A$ is an $N$-function, so is $\tA$. Furthermore, we have
\begin{equation}\label{prel:double_complementary}
    \ttA = A.
\end{equation}

Besides the Luxemburg functional, there are other ways that an Orlicz space can be equivalently normed. An important example for us will be the \emph{Orlicz functional}, which enables us to formulate a sharp H\"{o}lder-type inequality in Orlicz spaces. Given a Young function $A$, the corresponding Orlicz functional $\|\cdot\|_{\lAOrl}$ is defined as
\begin{equation*}
   \|x\|_{\lAOrl} =  \sup_{\rtA(y)\leq 1} \sum_{n = 1}^\infty |x_n y_n|,\ x=\{x_n\}_{n=1}^\infty,
\end{equation*}
where the supremum extends over all sequences $y=\{y_n\}_{n=1}^\infty$ such that $\rtA(y)\leq1$. It is a norm on $\lA$, which is equivalent to the Luxemburg norm. More precisely, we have
\begin{equation}\label{prel:Orlicz_and_Lux_equiv}
    \|x\|_{\lA} \leq \|x\|_{\lAOrl} \leq 2 \|x\|_{\lA} \quad \text{for every $x=\{x_n\}_{n=1}^\infty$}.
\end{equation}
A sharp H\"{o}lder-type inequality in Orlicz spaces reads as
\begin{equation*}
    \sum_{n = 1}^\infty |x_n| |y_n| \leq \|x\|_{\lA} \|y\|_{\ltAOrl} \quad \text{for all $x=\{x_n\}_{n=1}^\infty$, $y=\{y_n\}_{n=1}^\infty$}.
\end{equation*}
Moreover, we have
\begin{equation}\label{prel:asoc_double_Lux}
    \|x\|_{\lA}  = \sup_{\|y\|_{\ltAOrl}\leq1} \sum_{n=1}^\infty |x_n||y_n|
\end{equation}
and
\begin{equation}\label{prel:asoc_double_Orl}
    \|x\|_{\lAOrl}  = \sup_{\|y\|_{\ltA}\leq1} \sum_{n=1}^\infty |x_n||y_n|
\end{equation}
for every $x=\{x_n\}_{n=1}^\infty$.

Situations in which a Young function $A$ is either linear near zero or equal to zero near $0$ are in a certain sense extremal. More precisely, it can be shown that $\lA = \ell_1$ (up to equivalent norms) if and only if there are $c, t_0\in(0, \infty)$ such that $A(t) = ct$ for every $t\in[0, t_0]$. Moreover, $\lA = \ell_\infty$ (up to equivalent norms) if and only if there is $t_0\in(0, \infty)$ such that $A(t) = 0$ for every $t\in[0, t_0]$. Furthermore, we have $A(t) = 0$ for every $t\in[0, t_0]$ and for some $t_0\in(0, \infty)$ if and only if $\tA(t) = ct$ for every $t\in[0, t_1]$ and for some $c,t_1\in(0, \infty)$.

Let $A$ and $B$ be Young functions. We say that \emph{$A$ dominates $B$ near zero}, and write $B\prec_0 A$, if there are $t_0\in(0,\infty)$ and a constant $c>0$ such that
\begin{equation*}
    B(t) \leq A(ct) \quad \text{for every $t\in[0, t_0]$}.
\end{equation*}
We say that \emph{$A$ and $B$ are equivalent near zero} if $A\prec_0 B$ and $B\prec_0 A$ simultaneously. If $A$ and $B$ are equivalent near zero, then $\lA = \ell_B$, up to equivalent norms. We have $\ell_B \hookrightarrow \lA$ if and only if $A\prec_0 B$.

Recall that the weak Lebesgue sequence space $\ell_{1,\infty}$ is defined as
\begin{equation*}
    \ell_{1, \infty} = \{ x = \{x_n\}_{n  = 1}^\infty: \|x\|_{\ell_{1, \infty}} < \infty\},
\end{equation*}
where
\begin{equation*}
    \|x\|_{\ell_{1, \infty}} = \sup_{t>0} t \#\{n\in\N: |x_n| > t\},\ x = \{x_n\}_{n  = 1}^\infty.
\end{equation*}
The weak Lebesgue sequence space $\ell_{1,\infty}$ endowed with $\|\cdot\|_{\ell_{1, \infty}}$ is a quasi-Banach sequence space, which is not normable. 
We have $\ell_1 \hookrightarrow \ell_{1, \infty}$, and the inclusion is strict. The quasi-norm $\|\cdot\|_{\ell_{1, \infty}}$ can alternatively be expressed as
\begin{equation}\label{E:weak_l1_in_terms_of_rearrangement}
    \|x\|_{\ell_{1, \infty}} = \sup_{n\in\N} n x^*_n \quad \text{for every $x = \{x_n\}_{n  = 1}^\infty$}.
\end{equation}

For every $k\in\N$, we denote by $e^k$ the $k$th canonical sequence, that is, $(e^k)_n = \delta_{k,n}$ for $n\in\N$. Furthermore, for a bounded set $M\subseteq(0,\infty)$ with $M\cap \N\neq \emptyset$, we define its \emph{characteristic sequence} $\chi_M$ as
\begin{equation*}
    \chi_M = \sum_{k\in\N\cap M} e^k.
\end{equation*}
We have
\begin{align}
    \|\chi_M\|_{\lA} &= \frac1{A^{-1}\Big( \frac1{\# \N\cap M} \Big)} \nonumber\\
    \intertext{and}
    \|\chi_M\|_{\lAOrl} &= \tA^{-1}\Big( \frac1{\# \N\cap M} \Big)\# \N\cap M, \label{prel:Orlicz_norm_char_seq}
\end{align}
in which $A^{-1}$ and $\tA^{-1}$ are the right-continuous inverses of $A$ and $\tA$, respectively. The right-continuous inverse of a nondecreasing function $F\colon[0,\infty] \to [0,\infty]$ satisfying $F(0) = 0$ is defined as
\begin{equation*}
    F^{-1}(t) = \sup\{\tau \geq 0: F(\tau) \leq t\},\ t\in[0, \infty].
\end{equation*}
When a Young function $A$ is positive and finite-valued on the interval $(0, \infty)$, $A^{-1}$ is its classical inverse.

We recall that an operator $T$  whose domain is some linear space of sequences and whose range is contained in the set of all sequences is called \emph{quasi-linear} if there exists a constant $c\ge1$ such that the relations
    \begin{equation}
        \label{E:quasilinear}
        |(T(x+y))_n|
        \le c
        \left(
        |(Tx)_n|
        +
        |(Ty)_n|
        \right),
        \quad
        |T(\lambda x)|=|\lambda|\cdot|Tx|
    \end{equation}    
hold for all sequences $x$ and $y$ in the domain of $T$ and all $\lambda\in\R$. If~\eqref{E:quasilinear} holds with $c=1$, then $T$ is said to be \emph{sublinear}.

\section{Proofs}

We begin by noticing that every Young function $A\in\DTwoZ$ can be (non-uniquely) redefined outside some neighborhood of $0$ in such a way that the resulting Young function satisfies the $\DTwo$ condition globally and respective complementary functions coincide near $0$.
\begin{lemma}\label{lem:better_Young_function}
Let $A\in\DTwoZ$ be a Young function and let $\delta\in(0, \infty)$ be such that $A$ is positive and finite-valued on the interval $(0, 2\delta)$. Then there is a Young function $B$ such that $B\in\DTwo$, $A(t)=B(t)$ for every $t\in[0,\delta]$,
\begin{equation}\label{E:better_Young_function:not_linear_near_infty}
    \lim_{t\to\infty}\frac{B(t)}{t} = \infty,
\end{equation} 
 and $\tA(t) = \tilde{B}(t)$ for every $t\in[0,\frac{A(\delta)}{\delta}]$. In particular, we have $\lA = \ell_B$ and $\ell_{\tA} = \ell_{\tilde{B}}$, up to equivalent norms.
\end{lemma}
\begin{proof}
    Let $a$ be the left-continuous density of $A$ satisfying \eqref{prel:orlicz-integral-representation-of-young-function}. Note that $0<a(\delta)<\infty$. We now define a function $B\colon[0,\infty]\to[0,\infty]$ as
\begin{equation*}
    B(t) = \int_0^t b(\tau) \dd\tau,\ t\in[0, \infty],
\end{equation*}
where $b\colon[0,\infty) \to [0, \infty)$ is defined as
\begin{equation*}
    b(t) = \begin{cases}
    a(t) \quad &\text{for $t\in[0, \delta]$},\\
    a(\delta)\frac{t}{\delta} \quad &\text{for $t\in[\delta, \infty]$}.
    \end{cases}
\end{equation*}
In can be easily verified that $B$ is a Young function that coincides with $A$ on the interval $[0, \delta]$ and satisfies $B\in\DTwo$ and \eqref{E:better_Young_function:not_linear_near_infty}.

Note that since the function $\tau\mapsto \frac{A(\tau)}{\tau}$ is nondecreasing on $(0,\infty)$, we have 
\begin{equation*}
    \tau t - A(\tau) = \tau\Bigg( t - \frac{A(\tau)}{\tau} \Bigg) \leq \tau\Bigg( t - \frac{A(\delta)}{\delta} \Bigg) \leq 0
\end{equation*}
for all $t\in[0,\frac{A(\delta)}{\delta}]$ and $\tau\in[\delta, \infty)$. It follows that
\begin{equation}\label{E:better_Young_function:1}
    \tA(t) = \sup_{\tau\in[0,\delta]} \big( \tau t - A(\tau) \big) \quad \text{for every $t\in[0,\tfrac{A(\delta)}{\delta}]$}.
\end{equation}
Furthermore, since $b \geq a\chi_{[0,\delta)} + a(\delta)\chi_{[\delta, \infty)}$, we have
\begin{align*}
    \tau t - B(\tau) &\leq \tau t - \big( A(\delta) + a(\delta)(\tau - \delta) \big) = \tau\Bigg(t - a(\delta) + \frac{ \delta a(\delta) - A(\delta)}{\tau} \Bigg) \\
    &\leq \tau \Bigg( t - \frac{A(\delta)}{\delta} \Bigg) \leq 0
\end{align*}
for all $t\in[0,\frac{A(\delta)}{\delta}]$ and $\tau\in[\delta, \infty)$. In the last inequality we have employed~\eqref{E:relation-between-a-and-A}.
Hence, using this and \eqref{E:better_Young_function:1}, we obtain
\begin{equation*}
    \tilde{B}(t) = \sup_{\tau\in[0,\delta]}\big( \tau t - B(t) \big) = \sup_{\tau\in[0,\delta]}\big( \tau t - A(t) \big) = \tilde{A}(t)
\end{equation*}
for every $t\in[0, \frac{A(\delta)}{\delta}]$.

Finally, since two Young functions that coincide near $0$ are clearly equivalent near $0$, it follows that $\lA = \ell_B$ and $\ell_{\tA} = \ell_{\tilde{B}}$, up to equivalent norms.
\end{proof}

\begin{proof}[Proof of Theorem~\ref{T:main}]

We shall subsequently verify the chains of implications
\begin{equation*}
    \ref{it:one_fund_seq}
    \Rightarrow
    \ref{it:delta20}
    \Rightarrow
    \ref{it:interpolation_modular}
    \Rightarrow
    \ref{it:interpolation_norm}
    \Rightarrow
    \ref{it:average_bdd}
    \Rightarrow
    \ref{it:maximal_bdd}
    \Rightarrow
    \ref{it:maximal_bdd_restricted}
    \Rightarrow
    \ref{it:one_fund_seq}
    \quad\text{and}\quad
    \ref{it:delta20}
    \Rightarrow
    \ref{it:every_fund_seq}
    \Rightarrow
    \ref{it:one_fund_seq}.
\end{equation*}
We begin by showing that \ref{it:one_fund_seq} implies \ref{it:delta20}. 
Let $\{t_n\}_{n=1}^\infty$ be the sequence from~\ref{it:one_fund_seq}.
Set
\begin{equation*}
    M = \sup_{n\in\N} \frac{\tA(2t_n)}{\tA(t_n)} \in (1, \infty).
\end{equation*}
We claim that there is $n_0\in\N$ such that
\begin{equation}\label{E:main:1}
   1 -  \frac{t_n - 2t_{n+1}}{t_n} (M-1) \geq \frac1{2} \quad \text{for every $n\geq n_0$}.
\end{equation}
Indeed, since the sequence $\{t_n\}_{n=1}^\infty$ satisfies \eqref{E2}, we have
\begin{equation*}
    \limsup_{n\to\infty}\Big( 1 - \frac{2t_{n+1}}{t_n} \Big) \leq 0.
\end{equation*}
It follows that
\begin{equation*}
    \liminf \Big( 1 - \frac{t_n - 2t_{n+1}}{t_n}(M-1) \Big) = 1- \limsup_{n\to\infty}\frac{t_n - 2t_{n+1}}{t_n}(M-1) \geq 1.
\end{equation*}
Hence there is $n_0\in\N$ such that \eqref{E:main:1} is true. Furthermore, we claim that
\begin{equation}\label{E:main:2}
    \tA(2t) \leq 2M^2 \tA(t) \quad \text{for every $t\leq t_{n_0}$},
\end{equation}
which will prove that $\tA\in\DTwoZ$. To this end, fix $t\in(0, t_{n_0}]$. In view of \eqref{E1}, there is $n\geq n_0$ such that
\begin{equation*}
    t\in [t_{n+1},t_n].
\end{equation*}
We now distinguish between two cases. First, assume that $t_n \leq 2t_{n+1}$. Using the monotonicity of $\tA$, we obtain
\begin{equation}\label{E:main:3}
    \tA(2t)  \leq \tA(2t_n) \leq M \tA(t_n)
    \leq M \tA(2t_{n+1})
    \leq M^2 \tA(t_{n+1}) \leq M^2 \tA(t).
\end{equation}
Second, assume that $t_n > 2t_{n+1}$. Note that this means that $t_n - 2t_{n+1} > 0$. Since $\tA$ is convex and $2t_{n+1} < t_n < 2t_n$, we have
\begin{equation*}
    \frac{\tA(t_n)  - \tA(2t_{n+1})}{t_n - 2t_{n+1}} \leq \frac{\tA(2t_n)  - \tA(t_n)}{t_n}.
\end{equation*}
Using this and the fact that $\tA(2t_n) \leq M \tA(t_n)$, we obtain
\begin{equation*}
    \tA(t_n) - \tA(2t_{n+1})\leq \frac{t_n - 2t_{n+1}}{t_n} (M-1)\tA(t_n).
\end{equation*}
Hence, we have
\begin{equation*}
    \tA(t_n) \Big(1 -  \frac{t_n - 2t_{n+1}}{t_n} (M-1)\Big) \leq \tA(2t_{n+1}).
\end{equation*}
Combining this with \eqref{E:main:1}, we arrive at
\begin{equation*}
   \frac{\tA(t_n)}{2} \leq \tA(t_n) \Big(1 -  \frac{t_n - 2t_{n+1}}{t_n} (M-1)\Big) \leq \tA(2t_{n+1}),
\end{equation*}
whence it follows that
\begin{equation*}
    \tA(t_n) \leq 2 \tA(2t_{n+1}).
\end{equation*}
Hence, using the monotonicity of $\tA$ once more, we obtain
\begin{equation}\label{E:main:4}
    \tA(2t)  \leq \tA(2t_n) \leq M \tA(t_n) \leq 2M \tA(2t_{n+1})
    \leq 2M^2 \tA(t_{n+1})
\leq 2M^2 \tA(t).
\end{equation}
Finally, in view of \eqref{E:main:3} and \eqref{E:main:4}, we see that \eqref{E:main:2} is true, which proves the implication.

Second, we verify that \ref{it:delta20} implies \ref{it:interpolation_modular}. Assume that $\tA\in\DTwoZ$. If
\begin{equation*}
    \lim_{t\to 0^+} \frac{\tA(t)}{t} > 0,
\end{equation*}
then it is easy to see that $\tA$ is equivalent to the Young function $t\mapsto t$ near $0$. Recall that we know that the limit exists thanks to the monotonicity of the function $t\mapsto \frac{\tA(t)}{t}$ on $(0,\infty)$. Consequently, $A(t) = 0$ near zero, whence it follows that $\lA = \ell_\infty$, up to equivalent norms. Therefore, there is nothing to prove if this is the case. From now on, we assume that
\begin{equation}\label{E:maintheorem:1}
    \lim_{t\to 0^+} \frac{\tA(t)}{t} = 0.
\end{equation}
In view of Lemma~\ref{lem:better_Young_function} applied to $\tA$, we may assume without loss of generality that $\tA\in\DTwo$ and
\begin{equation*}
    \lim_{t\to\infty} \frac{\tA(t)}{t} = \infty.
\end{equation*}
Furthermore, since we assume \eqref{E:maintheorem:1}, it follows that $\tA$ is an $N$-function. Consequently, so is $A$. Since $A$ is an $N$-function whose complementary function satisfies the $\Delta_2$ condition, by \cite[Theorem~2.2]{Gal:88} there are constants $M,\gamma>0$ such that
\eqref{E:interpolation:modular} is satisfied, which shows that \ref{it:interpolation_modular} is true. For the reader's convenience, we provide the proof of this interpolation result for sequence spaces. Using the fact that $A$ is an $N$-function whose complementary function satisfies the $\Delta_2$ condition, we have (e.g., see~\cite[Proposition~1.4]{Gal:88})
\begin{equation*}
    \inf_{t\in(0,\infty)} \frac{t a(t)}{A(t)} > 1,
\end{equation*}
where $a\colon[0, \infty) \to [0, \infty)$ is the right-continuous density of $A$ (recall~\eqref{prel:orlicz-integral-representation-of-young-function}). Hence, there is $\beta>1$ such that
\begin{equation}\label{E:maintheorem:2}
    \beta A(t) \leq t a(t) \quad \text{for every $t\in[0, \infty)$}.
\end{equation}
It follows that
\begin{equation}\label{E:maintheorem:3}
    \beta \int_0^t \frac{A(s)}{s^2} \dd s \leq \int_0^t \frac{a(s)}{s} \dd s \quad \text{for every $t\in[0, \infty)$}.
\end{equation}
We claim that
\begin{equation}\label{E:maintheorem:4}
    \int_0^t \frac{A(s)}{s^2} \dd s < \infty \quad \text{for every $t\in(0, \infty)$}.
\end{equation}
Indeed, given $t\in(0, 1)$, we integrate \eqref{E:maintheorem:2} with $t$ replaced by $s$ over the interval $(t, 1)$ to obtain
\begin{equation*}
    \int_t^1 \frac{\beta}{s} \dd s \leq \int_t^1 \frac{a(s)}{A(s)} \dd s,
\end{equation*}
whence
\begin{equation*}
    \log\Big( \frac1{t^\beta} \Big) \leq \log\Big( \frac{A(1)}{A(t)} \Big),
\end{equation*}
which together with the monotonicity of the logarithm shows
\begin{equation*}
    \frac{A(t)}{t^2} \leq A(1) t^{\beta-2} \quad \text{for every $t\in(0,1)$}.
\end{equation*}
It is now easy to see that this together with the monotonicity and finiteness of $A$ implies \eqref{E:maintheorem:4}. Furthermore, integrating by parts and using the fact that
\begin{equation*}
    \lim_{t\to 0^+} \frac{A(t)}{t} = 0,
\end{equation*}
we see that
\begin{equation}\label{E:maintheorem:5}
    \int_0^t \frac{a(s)}{s} \dd s  = \frac{A(t)}{t} + \int_0^t \frac{A(s)}{s^2} \dd s \quad \text{for every $t\in(0, \infty)$}.
\end{equation}
In particular, this combined with \eqref{E:maintheorem:4} shows that we also have
\begin{equation}\label{E:maintheorem:6}
    \int_0^t \frac{a(s)}{s} \dd s < \infty \quad \text{for every $t\in(0, \infty)$}.
\end{equation}
Hence, using \eqref{E:maintheorem:5}, \eqref{E:maintheorem:3}, and \eqref{E:maintheorem:6}, we arrive at
\begin{equation}\label{E:maintheorem:7}
    \int_0^t \frac{a(s)}{s} \dd s \leq \frac{\beta}{\beta - 1}\frac{A(t)}{t} \quad \text{for every $t\in(0, \infty)$}.
\end{equation}
Now, let $T$ be an operator satisfying the assumptions of~\ref{it:interpolation_modular} whose constant of quasilinearity (that is, the constant  from~\eqref{E:quasilinear}) is denoted by $c$. Let further $x = \{x_n\}_{n=1}^\infty\in \lA$. We may clearly assume that $\|x\|_{\lA} > 0$, for otherwise there is nothing to prove. Set
\begin{equation*}
    C = \max\{\|T\|_{\ell_1\to \ell_{1, \infty}}, \|T\|_{\ell_\infty \to \ell_\infty}, c\} < \infty.
\end{equation*}
For each $t\in(0,\infty)$, we define sequences $x_t$ and $x^t$ as
\begin{align*}
    \{(x_t)_n\}_{n = 1}^\infty &= \{x_n\chi_{\{k\in\N: |x_k| > \frac{t}{2C^2}\}}(n)\}_{n=1}^\infty \\
    \intertext{and}
    \{(x^t)_n\}_{n = 1}^\infty &= \{x_n\chi_{\{k\in\N: |x_k| \leq \frac{t}{2C^2}\}}(n)\}_{n=1}^\infty = x - x_t.
\end{align*}
Since for each $t\in(0, \infty)$ we have
\begin{align*}
    \#\{n\in\N: |(Tx)_n| > t\} &\leq \#\Big\{ n\in\N: |(Tx_t)_n| > \frac{t}{2C} \Big\} \\
    &\quad+ \#\Big\{ n\in\N: |(Tx^t)_n| > \frac{t}{2C} \Big\} \\
    &= \#\Big\{ n\in\N: |(Tx_t)_n| > \frac{t}{2C} \Big\} \\
    &\leq \frac{2C^2}{t} \|x_t\|_{\ell_1} = \frac{2C^2}{t} \sum_{n = 1}^\infty |x_n|\chi_{\{s\geq0: s<2C^2|x_n|\}}(t),
\end{align*}
we obtain, using \eqref{prel:Orlicz_mnodular_via_distribution} and \eqref{E:maintheorem:7},
\begin{align*}
    \sum_{n = 1}^\infty A(|Tx|_n) &= \int_0^\infty a(t) \#\{n\in\N: |(Tx)_n| > t\} \dd t \\
    &\leq 2C^2 \sum_{n = 1}^\infty |x_n| \int_0^{2C^2|x_n|} \frac{a(t)}{t} \dd t \\
    &\leq \frac{\beta}{\beta - 1} \sum_{n = 1}^\infty A(2C^2 |x_n|).
\end{align*}
Hence, \eqref{E:interpolation:modular} is valid with $M = \beta/(\beta - 1)$ and $\gamma = 2C^2$.

Next, the fact that \ref{it:interpolation_modular} implies \ref{it:interpolation_norm} is easy. Indeed, setting
\begin{equation*}
    \lambda = \gamma M \|x\|_{\lA} \in(0, \infty),
\end{equation*}
we use \eqref{E:interpolation:modular} to obtain
\begin{equation*}
    \rA\Big( \frac{Tx}{\lambda} \Big) \leq M \sum_{n = 1}^\infty A\Big( \frac{|x_n|}{M \|x\|_{\lA}} \Big) \leq \sum_{n = 1}^\infty A\Big( \frac{|x_n|}{\|x\|_{\lA}} \Big) \leq 1,
\end{equation*}
where we also used \eqref{prel:convexity_small_lambda} and \eqref{prel:inf_att_by_Lux_norm}. It follows that
\begin{equation*}
    \|Tx\|_{\lA} \leq \lambda = \gamma M \|x\|_{\lA} \quad \text{for every $x=\{x_n\}_{n=1}^\infty\in\lA$}.
\end{equation*}
In other words, $T$ is bounded on $\lA$.

Now, we show that \ref{it:interpolation_norm} implies \ref{it:average_bdd}. Since $P$ is clearly linear and satisfies $\|P\|_{\ell_\infty\to\ell_\infty} = 1$, we only need to observe that it is of weak $(1,1)$-type, which is easy. Indeed, using \eqref{prel:HL_on_P} and \eqref{E:weak_l1_in_terms_of_rearrangement} together with the monotonicity of $P(x^*)$, we have
\begin{equation*}
    \|Px\|_{\ell_{1,\infty}} \leq \|P(x^*)\|_{\ell_{1,\infty}}  = \sup_{n\in\N} n \frac1{n}\sum_{k = 1}^n x^*_k = \|x^*\|_{\ell_1} = \|x\|_{\ell_1}
\end{equation*}
for every $x=\{x_n\}_{n = 1}^\infty$. Therefore, assuming \ref{it:interpolation_norm}, the boundedness of $P$ on $\lA$ follows, yielding~\ref{it:average_bdd}.

Next, we show that \ref{it:average_bdd} implies \ref{it:maximal_bdd}. This can be proved in the same way as in the continuous case on $\rn$. Indeed, by \cite[Lemma~11]{BG-E:06}, which follows from its classical analogue for functions (see~\cite[Chapter~3, Theorem~3.8]{BS}), we have
\begin{equation}\label{E:maintheorem:discrete_R-H-W}
    C_1 (Mx)_n^* \leq (P(x^*))_n \leq C_2 (Mx)_n^* \quad \text{for all $n\in\N$ and $x=\{x_n\}_{n = 1}^\infty$},
\end{equation}
where $C_1$ and $C_2$ are absolute constants. Hence, using the first inequality in \eqref{E:maintheorem:discrete_R-H-W} and \eqref{prel:Lux_norm_ri}, we obtain
\begin{equation*}
    \|Mx\|_{\lA} \leq C_1^{-1} \|P(x^*)\|_{\lA} \leq C_1^{-1}\|P\|_{\lA \to \lA} \|x^*\|_{\lA} = C_1^{-1}\|P\|_{\lA \to \lA} \|x\|_{\lA}
\end{equation*}
for every $x=\{x_n\}_{n = 1}^\infty$. Therefore, \ref{it:average_bdd} implies \ref{it:maximal_bdd}.

We finally show that~\ref{it:maximal_bdd_restricted} implies~\ref{it:one_fund_seq}. We may assume without loss of generality that $A$ grows faster than linearly near infinity (see the proof of Lemma~\ref{lem:better_Young_function}), which means that $\tA$ can be assumed to be finite everywhere. In particular, $\tA^{-1}$ is the classical inverse function to $A$ on $[0, \infty)$, which is continuous and strictly increasing. We first note that, assuming~\ref{it:maximal_bdd_restricted}, $\tA$ has to be positive outside the origin. Indeed, if $\tA(t)=0$ for some $t_0>0$ and all $t\in(0,t_0)$, then $A(t)\ge ct$ for some $c>0$ and $t_1>0$ and all $t\in(0,t_1)$. But then, the sequence $x=e^1$ clearly satisfies $x\in\ell_A$, while
\begin{equation*}
    \|Mx\|_{\ell_A}
    \ge 
    c\sum_{k=n_0}^{\infty}\frac{1}{k}
    =\infty,
\end{equation*}
where $n_0\in\N$ is chosen in such a way that $\frac{1}{n_0}\le t_1$.
We shall now construct the sequence $\{t_n\}_{n=1}^\infty$ with the properties required by~\ref{it:one_fund_seq}. To this end, fix some $m\in\N$ satisfying
\begin{equation}
    \label{E:log-estimate}
   m\ge e^{2C},
\end{equation}
where $C$ is the constant from \eqref{E:maximal_bdd_restricted},
and denote
\begin{equation*}
    t_n
    =\tA^{-1}\left(\frac{1}{mn}\right)
    \quad\text{for $n\in\N$.}
\end{equation*}
Then, $\{t_n\}_{n=1}^\infty$ is obviously a decreasing sequence converging to zero.
Moreover, by convexity, one has, for every $n\in\N$,
\begin{equation}
    \label{E:lower-bound-tA}
   \tA(2t_{n+1})
   \ge 2\tA(t_{n+1})
   =
   \frac{2}{m(n+1)}\ge\frac{1}{mn}.
\end{equation}
On applying the increasing function $\tA^{-1}$ to both sides of~\eqref{E:lower-bound-tA}, we get
\begin{equation*}
   2t_{n+1}
   \ge t_n
   \quad\text{for every $n\in\N$.}
\end{equation*}
In conclusion, our sequence $\{t_n\}_{n=1}^\infty$ satisfies~\eqref{E2}. It remains to verify that it also satisfies~\eqref{E3}. To this end, fix $n\in\N$ and consider the sequence $x=\chi_{[1,n]}$. Then, for $k\in\N$, one has
\begin{equation*}
    (Px)_k=
    \begin{cases}
        1&\text{if $k\le n$}
            \\
        \frac{n}{k}&\text{if $k\ge n$.}        
    \end{cases}
\end{equation*}
Furthermore, by the definition of the Orlicz norm and~\eqref{E:log-estimate}, we have
\begin{align*}
    \|Px\|_{\lAOrl}
    &\ge
    \sum_{k=1}^{\infty}
    (Px)_k\tA^{-1}\left(\frac{1}{mn}\right)\left(\chi_{[1,mn]}\right)_k
        \\
    &\ge 
    \tA^{-1}\left(\frac{1}{mn}\right)
    \sum_{k=n}^{mn}\frac{n}{k}
    =
    n\tA^{-1}\left(\frac{1}{mn}\right)
    \sum_{k=n}^{mn}\frac{1}{k}
        \nonumber\\
    &\ge
    n\tA^{-1}\left(\frac{1}{mn}\right)
    \int_{n}^{mn}\frac{dx}{x}  
    =
    n\tA^{-1}\left(\frac{1}{mn}\right)
    \log m    
        \nonumber\\
    &\ge
    2Cn\tA^{-1}\left(\frac{1}{mn}\right).\nonumber
\end{align*}
On the other hand, using the obvious inequality $(Px)_n\leq (Mx)_n$ for every $n\in\N$, \ref{it:maximal_bdd_restricted}, and \eqref{prel:Orlicz_norm_char_seq}, we obtain
\begin{equation*}
    \|Px\|_{\lAOrl} \leq \|Mx\|_{\lAOrl}
    \le C\|x\|_{\lAOrl}
    = Cn\tA^{-1}\left(\frac{1}{n}\right).
\end{equation*}
So, altogether,
\begin{equation*}
    2\tA^{-1}
    \left(\frac{1}{mn}\right)
    \le
    \tA^{-1}\left(\frac{1}{n}\right),
\end{equation*}
or, in other words,
\begin{equation}
    \label{E:dyadic-bound-2}
    2t_n
    \le
    \tA^{-1}\left(\frac{1}{n}\right).
\end{equation}
Applying the increasing function $\tA$ to both sides of~\eqref{E:dyadic-bound-2}, we get
\begin{equation*}
    \tA(2t_n)
    \le
    \frac{1}{n} = m \tA(t_n).
\end{equation*}
This establishes~\eqref{E3} for the sequence $\{t_n\}_{n=1}^\infty$ and thus finishes the proof of the implication \ref{it:maximal_bdd_restricted}$\Rightarrow$\ref{it:one_fund_seq}. Since the remaining implications~\ref{it:maximal_bdd}$\Rightarrow$\ref{it:maximal_bdd_restricted}
and~\ref{it:delta20}
$\Rightarrow$\ref{it:every_fund_seq}$\Rightarrow$\ref{it:one_fund_seq} are trivial, the proof is complete.
\end{proof}

\begin{proof}[Proof of Theorem~\ref{T:fundamental-sequence}] 
The `if' part of the assertion follows from the implication 
$\ref{it:one_fund_seq}
\Rightarrow
\ref{it:delta20}
$ of Theorem~\ref{T:main}. We shall thus prove the `only if' part.

Assume that~\eqref{E2} is not true. Then one has $K>2$, where 
\begin{equation*}
    K=\limsup_{n\in\N} \frac{t_n}{t_{n+1}}.
\end{equation*}
Suppose first that $\{t_n\}_{n=1}^\infty$ has the additional property that 
\begin{equation}
    \label{E:extra}
    \frac{t_n}{t_{n+1}}>\varrho
\end{equation}
for some $\varrho>2$ and every $n\in\N$. With no loss of generality, let us assume that $t_1=1$. Set
\begin{equation*}
    a(t)=
    \sum_{n=1}^{\infty}
    \frac{1}{n!}\chi_{(2t_{n+1},2t_n]}(t)
    +2t\chi_{(2,\infty)}(t)
    \quad\text{for $t\in(0,\infty)$}
\end{equation*}
and $a(0)=0$. Then $a$ is a left-continuous nondecreasing function on $(0,\infty)$ with $a(0+)=0$, which is finite and positive outside the origin. Consequently, the function $A$, defined as
\begin{equation*}
    A(t)=
    \int_{0}^{t}a(s)\dd s
    \quad\text{for $t\in[0,\infty)$,}
\end{equation*}
is a positive (outside the origin) Young function (in fact, an $N$-function). Set
\begin{equation*}
    \alpha=
    \frac{2}{\varrho}.
\end{equation*}
Then $\alpha\in(0,1)$. Since
\begin{equation*}
    2t_{n+1} <\alpha t_n<t_n<2t_n
    \quad\text{for $n\in\N$,}
\end{equation*}
we have $a(\alpha t_n)=a(2t_n)$. Thus, owing to the monotonicity of $a$ and the assumption~\eqref{E:extra}, we have
\begin{align}
    \label{E:upper}
    \frac{A(2t_n)}{A(t_n)}
    &=
    1+\frac{\int_{t_n}^{2t_n}a(s)\dd s}{\int_{0}^{t_n}a(s)\dd s}
        \le
    1+\frac{\int_{t_n}^{2t_n}a(s)\dd s}{\int_{\alpha t_n}^{t_n}a(s)\dd s}
        \\
        &\le 
     1+\frac{t_na(2t_n)}{(1-\alpha)t_na(\alpha t_n)}
    =\frac{2-\alpha}{1-\alpha}.
    \nonumber
\end{align}
Hence, the sequence $\{\frac{A(2t_n)}{A(t_n)}\}_{n=1}^\infty$ is bounded.
On the other hand, since, for $n\in\N$, $n\ge2$, one has
\begin{equation*}
    2t_{n} <4t_{n}<2\alpha t_{n-1}<2 t_{n-1},
\end{equation*}
we have
\begin{equation*}
    a(s)=a(2t_{n-1})
    \quad\text{for every
    $s\in(2t_n, 4t_n)$.}
\end{equation*}
Therefore,
\begin{align*}
    \frac{A(4t_n)}{A(2t_n)}
    &=
    1+\frac{\int_{2t_n}^{4t_n}a(s)\dd s}{\int_{0}^{2t_n}a(s)\dd s}
        \ge
    1+\frac{2t_na(2t_{n-1})}{2t_na(2t_n)}=1+n.
\end{align*}
Hence $A\notin\DTwoZ$. In combination with~\eqref{E:upper}, this shows that $\{t_n\}_{n=1}^\infty$ is not $\DTwoZ$-fundamental.

What remains is to chip off the additional assumption~\eqref{E:extra}. Assume, thus, that we still have $K>2$, but~\eqref{E:extra} is not necessarily satisfied. 
Fix some $\varrho\in(2,K)$ and find a strictly increasing sequence of natural numbers $\{n_k\}_{k=1}^\infty$ such that
\begin{equation}
    \label{E:limsup-restricted}
    \frac{t_{n_k}}{t_{n_k+1}}\ge\varrho.
\end{equation}
Owing to the definition of the limes superior, such a sequence has to exist. Put
\begin{align*}
    a(t)&=
    \sum_{k=1}^{\infty}
    \frac{2t_{n_1}}{k!}
    \chi_{(2t_{n_{k+1}+1}
    ,2t_{n_k+1}]}(t)
    +2t_{n_1}
    \chi_{(2t_{n_1+1},2t_{n_1}]}(t)
    +t\chi_{(2t_{n_1},\infty)}(t)
\end{align*}
for $t\in(0,\infty)$ 
and $a(0)=0$. Then $a$ is a left-continuous nondecreasing function on $(0,\infty)$ with $a(0+)=0$, which is finite and positive outside the origin. Consequently, the function $A$, defined as
\begin{equation*}
    A(t)=
    \int_{0}^{t}a(s)\dd s
    \quad\text{for $t\in[0,\infty)$,}
\end{equation*}
is, once again, a positive (outside the origin) Young function (or an $N$-function for that matter). Now it is easy to show (more or less verbatim as above) that 
\begin{equation*}
    \sup_{k\in\N}
    \frac{A(2t_{n_k})}{A(t_{n_k})}
    <\infty,
\end{equation*}
while
\begin{equation*}
    \sup_{k\in\N}
    \frac{A(4t_{n_k})}{A(2t_{n_k})}
    =\infty.
\end{equation*}
What remains is to show that
\begin{equation*}
    \sup_{n\in\N}
    \frac{A(2t_{n})}{A(t_{n})}
    <\infty.
\end{equation*}
To this end, fix any $j> n_1$ such that $j\neq n_k$ for any $k\in\N$. Then there is a uniquely determined $m\in\N$ such that $n_{m}<j<n_{m+1}$. 
Hence,
\begin{equation*}
    t_{n_{m+1}+1}<t_{n_{m+1}}
    <t_j<t_{n_m},
\end{equation*}
and, owing to~\eqref{E:limsup-restricted},
\begin{equation*}
    t_j>t_{n_{m+1}}\ge\varrho
    \cdot  t_{n_{m+1}+1}.
\end{equation*}
This together with $j\geq n_m+1$ implies that
\begin{equation*}
    2t_{n_{m+1}+1}
    \le
    \frac{2}{\varrho}t_j
    <t_j
    <2t_j
    \le
    2t_{n_m+1}.
\end{equation*}
Owing to the definition of $a$, this shows that $a$ is constantly equal to $a(2t_{n_m+1})$ on the interval $(2t_{n_{m+1}+1},2t_j)$.
Moreover,
\begin{equation*}
    0<\frac{t_j}{t_j-2t_{n_{m+1}+1}}
    =
    1+
    \frac{1}{\frac{t_j}{2t_{n_{m+1}+1}}-1}
    \le
    1+\frac{1}{\frac{\varrho}{2}-1}=\frac{\varrho}{\varrho-2}.
\end{equation*}
Altogether, we arrive at
\begin{align*}
    \frac{A(2t_j)}{A(t_j)}
    &=
    1+\frac{\int_{t_j}^{2t_j}a(s)\dd s}
    {\int_{0}^{t_j}a(s)\dd s}
        \le
    1+\frac{\int_{t_j}^{2t_j}a(s)\dd s}{\int_{2t_{n_{m+1}+1}}^{t_j}a(s)\dd s}
        \\
        &=
     1+
     \frac{t_ja(2t_j)}{(t_j-2t_{n_{m+1}+1})a(2t_{j})}
   \le
    1+\frac{\varrho}{\varrho-2}.
    \nonumber
\end{align*}
Finally, we have to consider $j<n_1$. For any such $j$, one has
either $t_j>4t_{n_1}$, in which case one has
\begin{align*}
    \frac{A(2t_j)}{A(t_j)}
    &=
    1+\frac{\int_{t_j}^{2t_j}a(s)\dd s}
    {\int_{0}^{t_j}a(s)\dd s}
        \le
    1+\frac{\int_{t_j}^{2t_j}s\dd s}{\int_{\frac12t_{j}}^{t_j}s\dd s}
       =5,
    \nonumber
\end{align*}
or $t_j\le4t_{n_1}$, in which case one has
\begin{equation*} 
    A(2t_j)\le A(8t_{n_1})\le
    8t_{n_1}a(8t_{n_1})
    \le 64 t_{n_1}^2,
\end{equation*}
while, using $t_j>t_{n_1}$ and \eqref{E:limsup-restricted},
\begin{equation*}
    A(t_j)\ge 
    \int_{t_{n_1+1}}^{t_{n_1}}a(s)\dd s
    =
    2t_{n_1}
    \left(t_{n_1}
    -t_{n_1+1}\right)
    \ge
    \frac{2(\varrho-1)}{\varrho}
    t_{n_1}^2,
\end{equation*}
hence, once again
\begin{align*}
    \frac{A(2t_j)}{A(t_j)}
    \le
    \frac{32\varrho}{\varrho-1}.
\end{align*}
So the sequence $\{\frac{A(2t_n)}{A(t_n)}\}_{n=1}^\infty$ is bounded uniformly in $n$, and the proof is complete.
\end{proof}

\begin{proof}[Proof of Theorem~\ref{T:main_dual}]
    We start by showing that \ref{it:average_bdd_dual} and \ref{it:delta20_dual} are equivalent. To this end, applying Theorem~\ref{T:main} to $\tA$ and using \eqref{prel:double_complementary}, we see that $A\in\DTwoZ$ if and only if the averaging operator $P$ is bounded on $\ltA$. Note that
    \begin{align*}
        \sup_{\|x\|_{\ltA}\leq1, x\geq0}\|Px\|_{\ltA} &= \sup_{\substack{\|x\|_{\ltA}\leq1, x\geq0\\ \|y\|_{\lAOrl}\leq1, y\geq0}} \sum_{n=1}^\infty (Px)_n y_n \\
        &= \sup_{\substack{\|x\|_{\ltA}\leq1, x\geq0\\ \|y\|_{\lAOrl}\leq1, y\geq0}} \sum_{n=1}^\infty \Big( \frac1{n} \sum_{k = 1}^n x_k \Big) y_n \\
        &= \sup_{\substack{\|x\|_{\ltA}\leq1, x\geq0\\ \|y\|_{\lAOrl}\leq1, y\geq0}} \sum_{k = 1}^\infty x_k \sum_{n=k}^\infty \frac{y_n}{n} \\
        &= \sup_{\|y\|_{\lAOrl}\leq1, y\geq0} \|Qy\|_{\lAOrl}
    \end{align*}
    thanks to \eqref{prel:asoc_double_Lux}, \eqref{prel:asoc_double_Orl}, and \eqref{prel:double_complementary}. It follows that $P$ is bounded on $\ltA$ if and only if $Q$ is bounded on $\lA$ (recall \eqref{prel:Orlicz_and_Lux_equiv}). Hence, \ref{it:average_bdd_dual} and \ref{it:delta20_dual} are equivalent.

    Next, it is easy to see that \ref{it:interpolation_dual} implies \ref{it:average_bdd_dual}. Indeed, since $Q|x|$ is a nonincreasing sequence for every sequence $x$, we have
    \begin{equation*}
        \sum_{k=1}^N (Qx)_k^* \leq \sum_{k=1}^N (Q|x|)_k^* = \sum_{k=1}^N (Q|x|)_k \quad \text{for every $N\in\N$},
    \end{equation*}
    which is \eqref{E:interpolation_dual:HLP} with $T=Q$ and $C=1$. Therefore, the proof will be finished once we show that \ref{it:average_bdd_dual} implies \ref{it:interpolation_dual}. To this end, with a slight overload of notation, for a sequence $z=\{z_n\}_{n=1}^\infty$, let $z^*\colon[0, \infty) \to [0 ,\infty]$ be the function defined as
    \begin{equation*}
        z^*(t) = \inf\{\lambda>0: \#\{n\in\N: |z_n| > \lambda \}\leq t\},\ t\in[0, \infty).
    \end{equation*}
    Note that $z^*(t) = z^*_n$ for all $t\in[0,\infty)$ and $n\in\N$ satisfying $t\in[n-1, n)$. Now, let $T$ and $x=\{x_n\}_{n=1}^\infty$ be as in \ref{it:interpolation_dual}. For each $t\in[0, \infty)$ and the unique $n\in\N$ such that $t\in[n-1,n)$, we have
    \begin{align*}
        \int_0^t (Tx)^*(s) \dd s &= \sum_{k = 1}^{n-1} \int_{k-1}^k (Tx)^*(s) \dd s + \int_{n-1}^t (Tx)^*(s) \dd s \\
        &= \sum_{k=1}^{n-1} (Tx)_k^* + (t - n + 1)(Tx)_n^* \leq \sum_{k=1}^{n-1} (Tx)_k^* + (Tx)_{n-1}^* \\
        &\leq 2\sum_{k=1}^{n-1} (Tx)_k^* \leq 2C \sum_{k=1}^{n-1} (Q|x|)_k = 2C \int_0^{n-1} (Q|x|)^*(s) \dd s \\
        &\leq 2C \int_0^t (Q|x|)^*(s) \dd s
    \end{align*}
    thanks to \eqref{E:interpolation_dual:HLP} with $N = n-1$. 
    Hence, it follows from the so-called Hardy-Littlewood-P\'olya principle \cite[Chapter~2, Theorem~4.6]{BS} that
    \begin{equation*}
        \|Tx\|_{\lA} \leq 2C \|Q|x|\|_{\lA},
    \end{equation*}
    whence we obtain
    \begin{equation*}
        \|Tx\|_{\lA} \leq 2C\|Q\|_{\lA\to\lA} \||x|\|_{\lA} = 2C\|Q\|_{\lA\to\lA} \|x\|_{\lA} \quad \text{for every $x\in\lA$}.
    \end{equation*}
    In other words, $T$ is bounded on $\lA$, which finishes the proof. 
\end{proof}

\begin{proof}[Proof of Theorem~\ref{T:hilbert}]
    By~\cite[Theorem~3]{And:75},
    ~\eqref{E:h-on-orlicz} holds if and only if both 
    \begin{equation}\label{E:p-on-orlicz}
        P\colon \ell_A\to\ell_A
    \end{equation}
    and
    \begin{equation}\label{E:q-on-orlicz}
        Q\colon \ell_A\to\ell_A
    \end{equation}
    do.
    By Theorem~\ref{T:main},~\eqref{E:p-on-orlicz} holds if and only if $\tA\in\DTwoZ$, while~\eqref{E:q-on-orlicz} holds if and only if $A\in\DTwoZ$ by Theorem~\ref{T:main_dual}.  
\end{proof}


\end{document}